\documentclass[12pt]{amsart}

\usepackage[titletoc,toc, title]{appendix}

\usepackage{amsmath}
\usepackage{amsthm}
\usepackage{amssymb}
\usepackage{amscd}
\usepackage{hyperref}

\usepackage{color}

\newcommand{\mf}{\mathfrak}

\DeclareMathOperator{\lc}{H^0_{\mf m}}

\newcommand{\Ass}{\operatorname{Ass}}
\newcommand{\Ann}{\operatorname{Ann}}
\newcommand{\depth}{\operatorname{depth}}
\newcommand{\height}{\operatorname{ht}}
\newcommand{\Spec}{\operatorname{Spec}}

\newcommand{\Minh}{\operatorname{Minh}}

\newcommand{\length}{\ell}

\newcommand{\gr}{\operatorname{gr}}

\renewcommand{\phi}{\varphi}

\usepackage{chngcntr}

\newtheorem{theorem}{Theorem}
\newtheorem{lemma}[theorem]{Lemma}
\newtheorem{proposition}[theorem]{Proposition}
\newtheorem{corollary}[theorem]{Corollary}

\newtheorem*{statement*}{Statement}
\newtheorem*{theorem*}{Theorem}
\newtheorem*{lemma*}{Lemma}
\newtheorem*{fact*}{Fact}

\theoremstyle{definition}
\newtheorem{definition}[theorem]{Definition}
\newtheorem*{definition*}{Definition}

\newtheorem*{example*}{Example}

\theoremstyle{remark}
\newtheorem{remark}[theorem]{Remark}
\newtheorem*{remark*}{Remark}

\begin{document}
\title[Prime filtrations of the powers of an ideal]
{Prime filtrations of the powers of an ideal}

\author{Craig Huneke}
\author{Ilya Smirnov}
\address{Department of Mathematics\\
University of Virginia\\
 Charlottesville, VA 22904-4137 USA}

\date{\today}

\begin{abstract}
We prove that for all $n$, simultaneously, we can choose prime filtrations of $R/I^n$ 
such that the set of primes appearing in these filtrations is finite. 
\end{abstract}


\thanks{The first author is partially supported by NSF grant 1259142.}

\maketitle

\section{Introduction}

Let $R$ be a Noetherian ring, let $I$ be an ideal of $R$, and let $M$ be a finitely generated
$R$-module. In 1979, Markus Brodmann (\cite{Brodmann}) proved that the sets of associated primes of $M/I^nM$
stabilize for $n$ sufficiently large. In particular, the union of the associated primes of all $M/I^nM$ is a finite set. This result furthered results of Ratliff (\cite{Ratliff1}) proved in 1976, and has since been used by many authors.  In this paper we prove a result which
seems to have been overlooked: there is a finite set of primes such that for all $n$,
$M/I^nM$ has a prime filtration involving only primes in that finite set.  Moreover, in
the case in which $R$ has infinite residue fields, we prove that the
set can be chosen to be stable for large $n$. We recall
that a \it prime filtration \rm of an $R$-module $N$ is a filtration $0 = N_0 \subset N_1 \subset \ldots \subset N_n = N$ such that for all $i$, $N_i/N_{i-1}\cong R/P_i$ for
some prime $P_i$. By abuse of language, the set of all such $P_i$ are said to be the primes in the filtration.   

On the face of it, our result is a stronger result than that of Brodmann;
however, the results are not perfectly comparable. While it is true that all associated primes
are always among the primes in a prime filtration (thus our result does prove that the set
of associated primes of $M/I^nM$ is finite as $n$ varies), it is not clear that -- just because
the primes in a prime filtration of $M/I^nM$ stabilize -- also the set of associated primes
stabilizes.

We  also give an estimate on the number of times that a given prime appears in these special filtrations, 
and we use it to bound the length of the local cohomology modules, reproving a result of Ulrich and Validashti
(\cite{UlrichValidashti}).

We were motivated to prove such a result by the second author's research concerning the upper semi-continuity
of the Hilbert-Kunz multiplicity. At one point, it seemed necessary to prove that one could invert a single element
to make $R/I^n$ Cohen-Macaulay for all $n$. If $I= P$ is prime, one can obtain this result by applying generic flatness
to the associated graded ring of $P$, see Remark~\ref{genflatness}. In general, however, it is not clear.
We apply our result to prove results concerning the openness of loci where
all $M/I^nM$ are Cohen-Macaulay, and further generalize to more general types of
filtrations.  The method we use come from the theory of superficial elements. We include
an appendix which proves some results on superficial elements.  These results are
essentially folklore, but we could not find a reference for them in the generality we
need for this paper.

\medskip

\section{Main Results}

\medskip

\begin{definition}\label{filterdef}
Let $R$ be a Noetherian ring and $I$ an ideal of $R$. 
In this note, we say that a finite $R$-module $M$ is an $I$-filtered module
if it is endowed with a filtration $M = M_0 \supseteq M_1 \supseteq \ldots \supseteq M_n \supseteq \ldots $
such that
\begin{enumerate}
\item $IM_n \subseteq M_{n + 1}$ for all $n$,
\item $\oplus M_n$ is a finitely generated module over the Rees ring, $\mathfrak{R}(I): = \oplus I^n$.
\end{enumerate}
\end{definition}

\begin{remark}\label{filterprop}
It is worth remarking that an $I$-filtered module $M$ satisfies the following properties:
\begin{enumerate}
\item[(a)] $\gr(M) =\bigoplus_{n\geq 0} M_n/M_{n+1}$ is a finite $\gr_I(R)$-module,
\item[(b)] the filtrations $\{M_n\}$ and $\{I^nM\}$ are cofinal.
\end{enumerate}
The last condition means (under the assumption that $IM_n \subseteq M_{n + 1}$ for all $n$) 
that for any $n$ there exists $k_n$ such that $M_n \subseteq I^{k_n}M$ and $\lim\limits_{n \to \infty} k_n = \infty$.
\end{remark}

\begin{definition}
Let $R$ be a ring and $M$ be a module over $R$, with a prime filtration
\[
0 = M_0 \subset M_1 \subset \ldots \subset M_N = M.
\]
For a given prime prime ideal $P$, its multiplicity $\mu_P$ in the prime filtration $\{M_k\}$
is the number of quotients $M_k/M_{k-1}$ isomorphic to $R/P$.
Note that $\mu_P$ depends on the choice of filtration.
\end{definition}

\begin{remark}
Given a short exact sequence 
\[
0 \to L \to M \to N \to 0
\]
and prime filtrations $\{L_k\}$ and $\{N_k\}$ of $L$ and $N$ respectively, we can build a prime filtration of $M$
in the following way. 
Lift submodules $N_k$ to their preimages $N_k + L$ in $M$, then it is easy to check that 
\[
0 = L_0 \subset L_1 \subset \ldots \subset L_N = L = N_0 + L \subset N_1 + L \subset \ldots \subset M
\]
is a prime filtration of $M$. Also, note that for these fixed filtrations,
$\mu_P(M) = \mu_P (N) + \mu_P(L)$. 
\end{remark}

\begin{theorem}\label{filtrations}
Let $R$ be a Noetherian ring, $I$ be an ideal, and let $M$ be an $I$-filtered module. 
Then there exists a finite set of prime ideals $\Lambda$ such that for any $n$ 
there exists a prime filtration of $M/M_n$ that consists only of prime ideals in $\Lambda$.
For these filtrations we can estimate the number of times that any given prime in $\Lambda$ 
appears in the filtration of $M/M_n$ as $O(n^{\dim M})$.

Furthermore, if $R$ has infinite residue fields, we can choose prime filtrations with stabilizing
sets of prime factors, i.e. there exists a subset $\Lambda' \subseteq \Lambda$
such that for all $n$ sufficiently large $\Lambda'$ is precisely the set of prime factors
of the chosen filtration of $M/M_n$.

\end{theorem}
\begin{proof}
We prove all the claims by contradiction. There are some small differences in the argument in the second case in which $R$ has
infinite residue fields which we point out at the relevant points in the argument. 
By Noetherian induction there is a maximal submodule (under inclusion) $L$
such that the theorem is false in $M/L$ (for the induced filtration).  
Note that the theorem is trivially true for the zero module, so $L$ is a proper submodule.
Since quotients of $I$-filtered modules are also $I$-filtered, without loss of generality, we assume that $L = 0$. 
We reach the situation in which for every nonzero submodule $M'$ of $M$, the theorem holds with the induced filtration on $M/M'$.

By Proposition~\ref{supexists}, there is an integer $m\geq 1$ such that $M$ has a superficial element $x$ of order $m$. 
Moreover, if $R$ has infinite residue fields, $m$ can be taken to be $1$.
By Proposition~\ref{supcolon}, there exists an integer $N$ such that for any $n \geq N$ the sequence
\[
0 \to M/((0 :_M x) + M_{n-m})\to M/M_n \to M/(M_n + xM)\to 0
\]
is exact.
By our assumption, the theorem
 holds in $M/xM$. 
Hence, if $R$ has infinite residue fields, then there exists a finite set of prime ideals $\Lambda_1$ 
and an integer $N_1$, such that for all $n \geq N_1$,
$M/(xM + M_n)$ has a prime filtration with the set of factors $\Lambda_1$.
Otherwise, set $N_1 = 0$ and let $\Lambda_1$ be a finite set of primes such that, for any $n$, 
$M/(xM + M_n)$ has a prime filtration with factors from $\Lambda_1$.
Also, we can choose these filtrations and a constant $C > 0$
such that the multiplicity for every prime in $\Lambda_1$ 
in the chosen filtration of  $M/(M_n + xM)$ is at most $C n^{\dim M/xM}$.
 
If $x$ is a zerodivisor on $M$, then $0:_M x \neq 0$, hence the assertion is true in $M/(0:_M x)$.
Thus there exists a finite set of prime ideals $\Lambda_2$ and a constant $D > 0$ such that 
$M/((0 :_M x) + M_{n-m})$ has a prime filtration of the required form for any $n$
(or $n$ sufficiently large for the second part),
and the multiplicity of the appearing primes is bounded by $Dn^{\dim M/(0:_M x)}$.
Gluing the filtrations of $M/((0 :_M x) + M_{n-m})$ and $M/(M_n + xM)$, we obtain 
a prime filtration of $M/M_n$ with all factors in the finite set of primes $\Lambda = \Lambda_1 \cup \Lambda_2$
of multiplicities at most $Cn^{\dim M/xM} + Dn^{\dim M/(0:_M x)} \leq (C + D) n^{\dim M}$.
Also, if the prime factors of the filtrations of $M/((0 :_M x) + M_{n-m})$ and $M/(M_n + xM)$
stabilize, the glued filtrations will have the same property.

Otherwise, if $0 :_M x = 0$, 
then choose arbitrary prime filtrations of $M/M_{N + i}$, for $i = 0\ldots m-1$,
and let $\Lambda$ be the union of  $\Lambda_1$ and all prime factors appearing in these filtrations.
(Here $N$ and $\Lambda_1$ are as in the paragraphs above.)
Using the exact sequence
\[
0 \to M/M_{n-m}\to M/M_n \to M/(M_n + xM)\to 0
\]
and induction on $n$, one can easily see that for any $n \geq N$, 
$M/M_n$ has a prime filtration with the set of factors in $\Lambda$.

If $m = 1$, we only need to choose an arbitrary prime filtration 
of $M/M_{K}$ where $K = \max(N,N_1)$. 
Then again, by induction on $n$, one obtains prime filtrations of $M/M_n$
consisting exactly of the prime factors of $M/M_K$ and $\Lambda_1$.
 
For these filtrations, we can count the multiplicity of any fixed prime in $\Lambda$ in the following way.
Let $\mu_M (n)$ and $\mu_{M/xM} (n)$ be  multiplicities of this prime in the filtrations of 
$M/M_n$ and $M/(xM + M_n)$ that we just obtained.
For $n >> 0$, let $n - K = dm + i$, where $i < m$ is the remainder, then, by the construction,
\begin{align*}
\mu_M (n) &= \mu_M (n-m) + \mu_{M/xM} (n) = \mu_M(n-2m) + \mu_{M/xM} (n-m) + \mu_{M/xM} (n) = \ldots \\ 
&= \mu_M (K + i) + \sum_{j = 0}^{d} \mu_{M/xM}(n - jm) \leq \mu_M (K + i) + C \sum_{j = 0}^{d} (n - jm)^{\dim M/xM}.
\end{align*}
Moreover, there exists a constant $C'$ such that 
$C \sum_{j = 0}^{d} (n - jm)^{\dim M/xM} \leq C' n^{\dim M/xM + 1}$.
But $\dim M/xM \leq \dim M - 1$ since $x$ is a regular element on $M$,
so $\mu_M (n)$ has the required asymptotic behavior.

For the first part of the claim, we have showed that for all $n \geq N$ 
$M/M_n$ has a prime filtration with all factors from $\Lambda$.
But then the claim follows, since we can choose arbitrary prime filtrations of $M/M_n$ for $n < N$
and add their prime factors to $\Lambda$.
\end{proof}

\begin{remark}
In the general case, the proof above can be used to show that we can choose the filtrations that have
the sets of prime factors stabilizing periodically, i.e. there are finitely many finite sets 
$\Lambda_1, \ldots, \Lambda_m \subseteq \Spec R$ such that, for some $N \geq 0$ and all $i \geq 0$, 
$\Lambda_k$ is exactly the set of prime factors of $M/M_{N + ki}$ where $k = 1\ldots m$. 
\end{remark}

\begin{remark}
To appreciate the theorem better, let us give an example of prime filtrations with an infinite set of prime factors.

Consider $R = k[x,y]$ and $I = (x)$ (or, even, $I = 0$). 
For every $n > 0$ let $f_n$ be a nonzero element of $k[y]$. Then we can embed 
$R/(x)$ into $R/(x^n)$ by mapping $1 \mapsto f_nx^{n-1} + (x^n)$ and obtain an exact sequence
\[
0 \to R/(x) \to R/(x^n) \to R/x^{n-1}(f_n, x) \to 0.
\] 
Thus any minimal prime of $(f_n, x)$ is an associated prime of $R/x^{n-1}(f, x)$, 
so we can use it to build the filtration further.
Hence, we can choose $f_n$ to obtain infinitely many distinct minimal primes  
of $(f_n, x)$ and, thus, infinitely many prime factors.
\end{remark}

As a corollary, we recover the celebrated result of Ratliff (\cite{Ratliff1}).
\begin{corollary}
Let $R$ be a Noetherian ring and $I$ an ideal of $R$.
Then $\cup_n \Ass (I^n)$ is finite.
\end{corollary}

\begin{corollary}\label{powCM}
Let $R$ be an excellent ring and $I$ be an ideal of $R$.
Then there exists an element $f \notin \sqrt{I}$ such that 
$R_f/I^nR_f$ is Cohen-Macaulay for all $n$. 
\end{corollary}
\begin{proof}
By the theorem, we can choose prime filtrations of all $R/I^n$ 
such that are only finitely many primes $P_i$, $1 \leq i \leq l$, appearing in those filtrations.

Without loss of generality, let $P_1$ be a minimal prime of $I$.
Then we can invert an element $s \in \cap_{i \geq 2} P_i \setminus P_1$ 
to make $P_1$ be the only prime appearing in the prime filtrations.
Since $R/P_1$ is excellent, its Cohen-Macaulay locus is open (\cite[7.8.3(iv)]{EGA}), 
so we can further localize at an element $t$ outside of $P_1$ to make it Cohen-Macaulay. 
We claim that $R/I^n$ are Cohen-Macaulay in the localization by $f = st$.

Let $n$ be arbitrary and let $0 \subset M_1 \subset \ldots \subset R_f/I^nR_f$
be the prime filtration of $R_f/I^nR_f$ induced by the original filtration, 
so that all the quotients are isomorphic to $R_f/P_1R_f$.
Now, if $\mf q$ is an arbitrary prime ideal containing $P_1$, it is easy to prove by induction 
that $(M_k)_\mf q$ are Cohen-Macaulay.  
\end{proof}

\begin{remark}\label{genflatness}
If $I = P$ is a prime ideal, then one can easily deduce the corollary from Generic Freeness (\cite[22.A]{Matsumura}).
Namely, since $R/P$ is an excellent domain, we can invert an element and assume that it is regular.

Now $\gr_P (R)$  is a finitely generated $R/P$-algebra, so by Generic Freeness 
we can invert an element of $R/P$ and make it free over the regular ring $R/P$. 
It follows that $P^n/P^{n+1}$ are projective $R/P$-modules for all $n$.
Then, using the sequences 
\[
0 \to P^n/P^{n+1} \to R/P^{n+1} \to R/P^n \to 0,
\]
we get that all residue rings $R/P^n$ are Cohen-Macaulay in this localization.
\end{remark}

For an ideal $I$, let $\Minh (I)$ to be the set of minimal primes $P$ of $I$ such that 
$\dim R/P = \dim  R/I$.

\begin{corollary}\label{loceqCM}
Let $R$ be a locally equidimensional excellent ring and $I$ be an ideal of $R$.
Then there exists an element $f \notin \cup \Minh (I)$ such that 
$R_f/I^nR_f$ is Cohen-Macaulay for all $n$. 
\end{corollary}
\begin{proof}
By the theorem, we can choose prime filtrations of all $R/I^n$ such that 
there are only finitely many primes $P_i$, $1 \leq i \leq l$, appearing in those filtrations.
Without loss of generality, let $\{P_1, \ldots, P_k\} = \Minh(I)$.

By prime avoidance we can find an element $t \in \bigcap_{i = {k +1}}^l P_i \setminus \bigcup_{i = 1}^k P_i$.
Then the induced prime filtrations in $R_t$ contain only $\Minh(I)$ as prime factors.
For $1 \leq i \leq k$, let $J_i$ be a preimage in $R$ of an ideal defining the non Cohen-Macaulay locus of $R/P_i$,
so $\height J_i > \height P_i = \height I$. 

Let $J = J_1 \cdots J_k$. We claim that there exists $s \in J \setminus \cup^k_1 P_i$.
If not, then for some $i, j$ there would be a containment $J_i \subseteq P_j$. 
But this is impossible since $\height J_i > \height I = \height P_j$.
Now, we let $f = st$ and prove that $R_f/I^nR_f$ is Cohen-Macaulay for all $n$.

Let $n$ be arbitrary and let $0 \subset M_1 \subset \ldots \subset R_f/I^nR_f$
be the prime filtration of $R_f/I^nR_f$ induced by the original filtration.
Since $R$ is locally equidimensional, for any prime $Q$ containing $I$, 
$\Minh (IR_Q)$ consists of the primes in $\Minh(I)$ contained in $Q$.
So we may localize at $Q$ and assume $R$ is local.

We prove by induction that $\depth M_k = \dim R/I$.
The base case of $M_1 = R/P_i$ is clear.
Now, consider the sequence 
\[0 \to M_k \to M_{k+1} \to R/P_i \to 0
\]
and apply the induction hypothesis.
 \end{proof}

\begin{corollary}\label{grCM}
Let $R$ be an excellent ring and $I$ be an ideal of $R$.
Then there exists an element $f \notin \sqrt{I}$ such that the associated graded ring
$\gr_{I_f}(R_f)$ is a Cohen-Macaulay module over $R_f/I_f$. 
If $R$ is locally equidimensional, we can choose $f \notin \cup \Minh(I)$.
\end{corollary}
\begin{proof}
By Corollary~\ref{powCM} (or Corollary~\ref{loceqCM} if $R$ is locally equidimensional), 
we can invert an element $f \notin \sqrt I$ and make all $R/I^n$ Cohen-Macaulay.
Note, that there are exact sequences
\[
0 \to \frac{I^nR_f}{I^{n+1}R_f} \to R_f/I^{n+1}R_f \to R_f/I^nR_f \to 0,
\]
so all $\frac{I^nR_f}{I^{n+1}R_f}$ are Cohen-Macaulay and the assertion follows.
\end{proof}

\begin{corollary}
Let $R$ be an analytically unramified ring, $I$ be an arbitrary ideal of $R$,
and $M$ be a finite $R$-module.
Then there is a finite set of prime ideals $\Lambda$ such that for all $n$ the module 
$M/\overline{I^n}M$ has a prime filtration where all prime factors are in $\Lambda$.

Furthermore, if $R$ is a locally equidimensional excellent ring, 
then there exists an element $f \notin \cup \Minh(I)$ such that 
$R_f/\overline{I^nR_f}$ are Cohen-Macaulay for all $n$. 
\end{corollary}
\begin{proof}
We will show that $\overline{I^n}M$ satisfies the conditions of Definition~\ref{filterdef}.

The first condition holds, since $I^{n + 1} \subseteq I\overline{I^n} \subseteq \overline{I^{n+1}}$.
Since $R$ is analytically unramified, the ring $\bigoplus_{n \geq 0} \overline{I^n}$
is a finite algebra over the Rees algebra $R[It]$ (\cite[Corollary 9.2.1]{HunekeSwanson})
so the second condition holds.

The second part of the proof is same as in Corollary~\ref{loceqCM}. 
\end{proof}

In \cite[Theorem 1.4, Lemma 1]{Rees2}, Rees showed that  $R$ is analytically unramified 
if and only if $\{\overline{I^n}\}$ is cofinal with $\{I^n\}$.
Via property (b) of Remark~\ref{filterprop} this highlights the necessity of the assumption.

\begin{remark}
We should note that one could modify the proof of Proposition~\ref{supexists} in our Appendix
and get that a superficial element for $R$ with respect to a filtration $I_n$ will exist 
if the Rees ring of the filtration $\{\overline{I^n}\}$
is Noetherian. However, this is also equivalent for $R$ to be analytically unramified,
so this will not lead to a generalization of the corollary above.
To prove this we first record a lemma for which it is hard to find a proof in print. However,
see \cite{Cowsik, Eliahou, Huneke}. The proof we record here is found in \cite{Huneke}. 
\end{remark}

\begin{lemma}
Let $\oplus_{n\geq 0} I_n$ be a Noetherian non-negatively graded ring, where $I_n$ form a decreasing chain of ideals in
$R = I_0$. Then there exists an
integer $l$ such that for all $n\geq 1$, $(I_l)^n = I_{ln}$.
\end{lemma}

\begin{proof} The ideal generated by all positive degree elements is finitely generated, say with generators up to degree $k$. 
Hence for all $m\geq k$, $I_m = \sum I_1^{j_1}I_2^{j_2}\cdots I_k^{j_k}$, where the sum ranges over all nonnegative integers
$j_1,...,j_k$ satisfying $\sum_i ij_i = m$. Set $l = k\cdot k!$. 

We first claim that for $m\geq l$, $I_m = I_{m-k!}I_{k!}$.  For if
$\sum ij_i = m\geq k\cdot k!$, then for some $1\leq a\leq k$, $aj_a\geq k!$. Note that $q = k!/a$ is an integer. 
Therefore,
\[I_m = I_1^{j_1}\cdots I_k^{j_k} = 
I_a^{q}I_1^{j_1}\cdots I_a^{j_a-q}I_{a+1}^{j_{a+1}}\cdots I_k^{j_k} \subseteq I_{m-k!}I_{k!},\] 
and the opposite inclusion is obvious.

We finish by proving $I_l^n = I_{ln}$ by induction on $n$. However,  $I_{ln} = I_{ln-k!}I_{k!} = I_{ln-2k!}I_{2k!} = \cdots = I_{ln-k\cdot k!}I_{k\cdot k!}
= I_{(n-1)l}I_l$
by above. So, the induction hypothesis finishes the proof. \end{proof}

The following proposition should be well-known, however we could not find a reference for it.

\begin{proposition}
Let $R$ be a Noetherian ring and $I$ be an ideal.
Then the Rees algebra $S = \bigoplus_{n \geq 0} \overline{I^n}$ is Noetherian if and only if $R$ is analyticaly unramified.
\end{proposition}
\begin{proof}
If $R$ is analytically unramified, $S$ is module finite over the Rees algebra $R[It]$, so it is Noetherian.

Now, assume that $S$ is Noetherian. Then by \cite[Theorem~2.7, Corollary~4.5]{Ratliff2}
there exists $k$ such that for all $n \geq 1$, $\overline{I^{nk}} = (\overline{I^k})^n$.

Since $I^k$ is a reduction of $\overline{I^k}$, there exists $n_0$ such that 
$I^k(\overline{I^k})^n = (\overline{I^k})^{n+1}$ for all $n \geq n_0$.
Therefore, for any $n \geq n_0$
\[
\overline{I^{nk}} = I^{k(n - n_0)}(\overline{I^k})^{n_0} \subseteq I^{k(n - n_0)}.
\]

Now, let $m \geq 0$ be arbitrary. We divide $m = nk + r$, where the reminder $r < k$.
Hence
\[
\overline{I^{m}} \subseteq \overline{I^{nk}} \subseteq I^{k(n - n_0)} \subseteq 
I^{m - k(n_0 + 1)}.
\] 
Since $k(n_0 + 1)$ is a fixed number, we have shown that the filtrations $\{\overline{I^n}\}$ and $\{I^n\}$
are cofinal, thus, by \cite[Lemma 1]{Rees2}, $R$ is analytically unramified.
\end{proof}

Using the multiplicity estimates of the filtrations constructed by Theorem~\ref{filtrations},
we give a different proof of the existence of $\epsilon$-multiplicity introduced in \cite{UlrichValidashti}.
Here $\ell(M)$ denotes the length of $M$.

\begin{corollary}\label{epsilon}
Let $R$ be a Noetherian ring, $I$ be an ideal in $R$, and $M$ be an $I$-filtered module.
Then $\length (\lc (M/M_n)) = O(n^{\dim M})$. 

In particular, we have
\[
\epsilon(I, M) = \limsup_{n \to \infty} \frac{d!\length(\lc (M/I^nM))}{n^{\dim M}} < \infty. 
\]
\end{corollary}
\begin{proof}
Since $\lc(-)$ is a semi-additive functor, we obtain that 
for a prime filtration $0 = L_0 \subset \ldots \subset L_k \subset \ldots \subset L_N = M/M_n$
\[
\length (\lc (M/M_n)) = \length (\lc(L_N)) \leq 
\length (\lc (L_{N-1})) + \length (\lc(R/P_N)) \leq \ldots \leq \sum_{i = 1}^N  \length (\lc(R/P_i)).
\]
Thus, if we take the filtrations obtained by Theorem~\ref{filtrations}, we get 
\[
\length (\lc (M/M_n)) \leq \sum_{P \in \Lambda} \mu_P (M/M_n)\length (\lc(R/P))
\leq \sum_{P \in \Lambda} Cn^{\dim M}\length (\lc(R/P)) = C n^{\dim M}.
\]

\end{proof}


\specialsection*{Acknowledgements}
We would like to thank Hailong Dao who suggested Corollary~\ref{epsilon}.

\begin{appendix}
\counterwithin{theorem}{section}

\section{Superficial elements for filtered modules.}

This appendix contains results which are well-known, but, surprisingly, we could
not find a reference for the generality we need.  


\begin{definition}
Let $R$ be a ring, $I$ an ideal, and $M$ an $I$-filtered module.
We say that $x \in I^m$ is a superficial element for $M$ of order $m$, if 
there exists $c \in \mathbb N$ such that for all $n \geq c$, $(M_{n + m} :_M x) \cap M_c = M_n$.
\end{definition}

\begin{proposition}\label{supexists}
Let $R$ be a Noetherian ring, $I$ an ideal, and $M$ an $I$-filtered module.
Then $M$ has a superficial element of some order $m$.
Furthermore, if $R$ has infinite residue fields, then $m$ can be taken to be $1$.
\end{proposition}
\begin{proof}
The proof appears in \cite[Proposition~8.5.7]{HunekeSwanson}.

Let $0 = N_1 \cap \ldots \cap N_r$ be a primary decomposition of the zero submodule in $\gr(M)$.
For $i = 1\ldots r$, let $P_i$ be the associated prime of $\gr(M)/N_i$. 
Without loss of generality, $P_1, \ldots, P_s$ contain all elements of $\gr_I(R)$ of positive degree,
and $P_{s+1}, \ldots, P_r$ do not. 
Then there exists an integer $d$ such that $\gr_I(R)_{\geq d} \subseteq \Ann \gr(M)/N_i$ for all $i = 1, \ldots, s$.
Since $\gr(M)$ is a finitely generated $\gr_I(R)$-module, there exists a constant $n_0$ such that
$\gr(M)_{\geq n} \subseteq \gr_I(R)_{\geq n - n_0}\gr(M)$ for all $n \geq n_0$.
Thus there is a constant $c = d + n_0$ such that for all $i = 1,\ldots, s$
\[\gr(M)_{\geq c} \subseteq \gr_I(R)_{\geq d}\gr(M) \subseteq N_i.\]

By Prime Avoidance there exists a homogeneous element $h$ of positive degree in $\gr_I(R)$ 
that is not contained in any $P_i$ for $i > s$. Say $h = x + I^{m+1}$ for some $x \in I^m$.
If $R$ has infinite residue fields, this $m$ can be taken to be $1$.

Note that $M_n \subseteq (M_{n + m} :_M x) \cap M_c$ for any $n \geq c$. 
Suppose $n \geq c$ and there exists $y \in (M_n :_M x) \cap M_c \setminus M_{n-m}$.
Let $k$ be the largest integer such that $y \in M_k$. Then $c \leq k < n$.
In $\gr(M)$, $(x + I^{m+1}) \cdot (y + M_{k+1}) = 0$. Thus by the choice of $x + I^{m+1}$,
\[y+M_{k+1} \in N_{s+1} \cap \ldots \cap N_r.\]
By the choice of $y$, $y + M_{k+1} \in M_c \cap N_1 \cap \ldots \cap N_s$,  
hence, by the choice of $c$, $y + M_{k+1} = 0$, a contradiction with the choice of $k$. 

\end{proof}

\begin{proposition}\label{supcolon}
Let $R$ be a Noetherian ring, $I$ an ideal, and $M$ be an $I$-filtered module. 
Suppose $x$ is a superficial element for $M$ of order $m$, then 
$M_n :_M x = 0 :_M x + M_{n - m}$ for all sufficiently large $n$.
\end{proposition}
\begin{proof}
Since the filtration is cofinal with $\{I^nM\}$,
there exists $k_n$ such that $M_n \subseteq I^{k_n}M$.
Thus
\[x(M_n :_M x) = M_{n} \cap xM \subseteq I^{k_n}M \cap xM.\]
Now, by the Artin-Rees Lemma, there exists $e$ such that 
$I^{k_n}M \cap xM \subseteq I^{k_n - e}(xM)$.
Since $\{M_n\}$ is cofinal to $I^n M$ and $k_n$ grows without a bound as a function of $n$, 
$xI^{k_n - e}M \subseteq xM_c$ for $n$ sufficiently large. 
Thus, we obtain that $(M_n :_M x) \subseteq M_c + 0:_M x$.

Therefore, $(M_n :_M x) = (M_n :_M x) \cap (M_c + 0:_M x) = (M_n :_M x) \cap M_c + 0:_M x = 
M_{n - m} + 0:_M x$.
\end{proof}

\end{appendix}

\bibliographystyle{plain}
\bibliography{fpbib}

\begin{thebibliography}{10}

\bibitem{Brodmann}
M.~Brodmann.
\newblock Asymptotic stability of {${\rm Ass}(M/I^{n}M)$}.
\newblock {\em Proc. Amer. Math. Soc.}, 74(1):16--18, 1979.

\bibitem{Cowsik}
R.~C. Cowsik.
\newblock Symbolic powers and number of defining equations.
\newblock In {\em Algebra and its applications ({N}ew {D}elhi, 1981)},
  volume~91 of {\em Lecture Notes in Pure and Appl. Math.}, pages 13--14.
  Dekker, New York, 1984.

\bibitem{Eliahou}
S.~Eliahou.
\newblock Symbolic powers of monomial curves.
\newblock {\em J. Algebra}, 117(2):437--456, 1988.

\bibitem{EGA}
A.~Grothendieck.
\newblock \'{E}l\'ements de g\'eom\'etrie alg\'ebrique. {IV}. \'{E}tude locale
  des sch\'emas et des morphismes de sch\'emas. {II}.
\newblock {\em Inst. Hautes \'Etudes Sci. Publ. Math.}, (24):231, 1965.

\bibitem{Huneke}
C.~Huneke.
\newblock Unmixed ideals in 3-dimensional regular local rings.
\newblock In {\em Atas da 9a Escola de \'Algebra}, pages 5--24. Sociedade
  Brasileira de Matem\'atica, 1986.

\bibitem{HunekeSwanson}
C.~Huneke and I.~Swanson.
\newblock {\em Integral closure of ideals, rings, and modules}, volume 336 of
  {\em London Mathematical Society Lecture Note Series}.
\newblock Cambridge University Press, Cambridge, 2006.

\bibitem{Matsumura}
H.~Matsumura.
\newblock {\em Commutative algebra}, volume~56 of {\em Mathematics Lecture Note
  Series}.
\newblock Benjamin Cummings Publishing Co., Inc., Reading, Mass., second
  edition, 1980.

\bibitem{Ratliff1}
L.~J. Ratliff, Jr.
\newblock On prime divisors of {$I^{n},$} {$n$} large.
\newblock {\em Michigan Math. J.}, 23(4):337--352 (1977), 1976.

\bibitem{Ratliff2}
L.~J. Ratliff, Jr.
\newblock Notes on essentially powers filtrations.
\newblock {\em Michigan Math. J.}, 26(3):313--324, 1979.

\bibitem{Rees2}
D.~Rees.
\newblock A note on analytically unramified local rings.
\newblock {\em J. London Math. Soc.}, 36:24--28, 1961.

\bibitem{UlrichValidashti}
B.~Ulrich and J.~Validashti.
\newblock Numerical criteria for integral dependence.
\newblock {\em Math. Proc. Cambridge Philos. Soc.}, 151(1):95--102, 2011.

\end{thebibliography}

\end{document}